\documentclass[10pt,a4paper]{amsart}
\usepackage[latin1]{inputenc}
\usepackage{latexsym}
\usepackage{color,graphicx,shortvrb}
\usepackage{amsmath, amssymb}
\usepackage{amsfonts}
\usepackage[colorlinks, bookmarks=true]{hyperref}

\newtheorem{theorem}{Theorem}[section]
\newtheorem{lemma}[theorem]{Lemma}

\newtheorem{corollary}[theorem]{Corollary}

\theoremstyle{definition}
\newtheorem{definition}[theorem]{Definition}
\newtheorem{remark}[theorem]{Remark}

\author{J. M. Almira$^*$, Kh. F. Abu-Helaiel}
\title{A note on invariant subspaces and the solution of some classical functional equations}
\thanks{$^*$ Corresponding author}

\begin{document}

%Ultrametric Banach spaces, Approximation theory, Lethargy theorems, p-adic transcendental numbers.}

\subjclass[2010]{}

% 41A65, A1A25, 11J61, 11J81, 11K60.}

\begin{abstract} We study the continuous solutions of several classical functional equations by using the properties of the spaces of continuous functions which are invariant under some elementary linear trans-formations. Concretely, we use  that the sets of continuous solutions of certain equations are closed vector subspaces of $C(\mathbb{C}^d,\mathbb{C})$ which are invariant under affine transformations $T_{a,b}(f)(z)=f(az+b)$, or closed vector subspaces of $C(\mathbb{R}^d,\mathbb{R})$ which are translation and dilation invariant. These spaces have been recently classified  by Sternfeld and Weit,  and Pinkus, respectively, so that we use this information to  give a direct characterization of the continuous solutions of the corresponding functional equations. 
\end{abstract}

\maketitle

\markboth{J. M. Almira, Kh. F. Abu-Helaiel}{Invariant subspaces and some classical functional equations}
%\section{Introduction}

\section{Motivation}

We study the continuous solutions of several classical functional equations by using the properties of the spaces of continuous functions which are invariant under some elementary linear transformations.  This study was recently initiated in connection with Fr\'{e}chet's functional equation \cite{almira_invariantes}, \cite{almira_montel_vv} and, after these contributions, we thought that perhaps the theory of invariant subspaces could be also used in connection with other classical equations. In this paper we demonstrate that this connection exists. Concretely, we use  that the sets of continuous solutions of several classical functional equations are closed vector subspaces of $C(\mathbb{C}^d,\mathbb{C})$ which are invariant under ``affine'' transformations $T_{a,b}(f)(z)=f(az+b)$, or closed vector subspaces of $C(\mathbb{R}^d,\mathbb{R})$ which are translation and dilation invariant. These spaces have been recently classified   (\cite{sternfeld_weit}, \cite{pinkus_TDI}), so that we use this information to  give a direct characterization of the continuous solutions of the corresponding functional equations. The main tool for our proofs are, henceforth, the following two well known results:

\begin{theorem}[Sternfeld and Weit \cite{sternfeld_weit}] \label{SW} Assume that $V$ is a closed subspace of $C(\mathbb{C},\mathbb{C})$ and $T_{a,b}(V)\subseteq V$ for all $a,b\in\mathbb{C}$, where 
$T_{a,b}:C(\mathbb{C},\mathbb{C})\to C(\mathbb{C},\mathbb{C})$ is defined by $T_{a,b}(f)(z)=f(az+b)$. Then  $V=\overline{\mathbf{span}_{\mathbb{C}}\left(\bigcup_{k=1}^rA_{[(n_k,m_k)]}\right)}$ for a certain finite set of points $\{(n_k,m_k)\}_{k=1}^r\subseteq(\widehat{\mathbb{N}})^2$, where $\widehat{\mathbb{N}}=\mathbb{N}\cup\{+\infty\}$, $[(n,m)]=\{(\alpha,\beta)\in \widehat{\mathbb{N}}^2 :0\leq \alpha\leq n \text{ and } 0\leq \beta \leq m\}$ (here, $\alpha \leq +\infty$ means that $\alpha \in\mathbb{N}$) and, given 
$J\subseteq \mathbb{N}^2$, $A_J=\overline{\mathbf{span}_{\mathbb{C}}\{z^n\overline{z}^m: (n,m)\in J\}}$. 
\end{theorem}

\begin{theorem}[Pinkus \cite{pinkus_TDI}]  \label{P} Assume that $V$ is a closed subspace of $C(\mathbb{R}^d,\mathbb{R})$ and $S_{a,b}(V)\subseteq V$ for all $a,b\in\mathbb{R}^d$, where 
$S_{a,b}:C(\mathbb{R}^d,\mathbb{R})\to C(\mathbb{R}^d,\mathbb{R})$ is defined by $S_{a,b}(f)(x)=f(a\cdot x+b)$. Then  
$$V=\overline{\mathbf{span}_{\mathbb{R}}\left(\bigcup_{k=1}^rB_{[n_k]}\right)}$$ 
for a certain finite set of points $\{n_k=(n_{k,1},\cdots,n_{k,d})\}_{k=1}^r\subseteq(\widehat{\mathbb{N}})^d$, where $[n]=\{\alpha\in \mathbb{N}^d: \alpha\leq n\}$  and, given 
$J\subseteq \mathbb{N}^d$, $B_J=\overline{\mathbf{span}_{\mathbb{R}}\{x^n: n\in J\}}$. 

\end{theorem}

\begin{remark} In Theorem \ref{P}, we have used the following standard notation, which will be used also in section 3 of this paper: If $\alpha=(\alpha_1,\cdots,\alpha_d)\in\mathbb{N}^d$, $a=(a_1,\cdots,a_d), b=(b_1,\cdots,b_d),x=(x_1,\cdots,x_d)\in \mathbb{R}^d$, then $a\cdot x=(a_1x_1,\cdots,a_dx_d)$, $x+b=(x_1+b_1,\cdots,x_d+b_d)$, $x^{\alpha}=x_1^{\alpha_1}x_2^{\alpha_2}\cdots x_d^{\alpha_d}$ and $|\alpha|=\sum_{k=1}^d \alpha_k$. Furthermore, if $n=(n_1,\cdots,n_d),m=(m_1,\cdots,m_d)\in\widehat{\mathbb{N}}^d$, then $n\leq m$ means that $n_1\leq m_1,n_2\leq m_2,\cdots,n_d\leq m_d$.  Finally, $\Pi_{m}^{d}$ denotes the set of real polynomials in $d$ real variables with total degree $\leq m$ (when $d=1$ we write $\Pi_m$ instead of $\Pi_{m}^1$). \end{remark}

In section 2 we study the mean value equation of Kakutani-Nagumo-Walsh (see \cite{haruki1}, \cite{haruki2}, \cite{haruki3}, \cite{kakutani_nagumo}, \cite{walsh}) and a variation of this equation introduced by Haruki in \cite{haruki3}. In section 3 we give another proof of Frechet's original theorem \cite{frechet}, for the space of continuous real functions of several real variables. It is important to note that the proofs of Theorems \ref{SW} and \ref{P} are of elementary nature (they do not use technical results from complex analysis nor measure theory), so that all proofs in this paper are also elementary.

\section{Continuous solutions of Kakutani-Nagumo-Walsh and Haruky functional equations}

Let $\theta$ be any primitive $n$-th root of $1$ and $\eta$ be any primitive $2N$-th root of $1$. In \cite{haruki3} S. Haruki studied the solutions of the following functional equations:
\begin{equation}\label{KNW}
\frac{1}{N}\sum_{k=0}^{N-1}f(x+\theta^ky)=f(x) \text{ for all } x,y\in\mathbb{C} \text{ (Kakutani-Nagumo-Walsh equation)}
\end{equation} 
\begin{equation}\label{N}
\sum_{k=0}^{N-1}(|f(x+\theta^ky)|^2-|f(x)|^2)=\sum_{k=0}^{N-1}|f(x+\theta^ky)-f(x)|^2 \text{ for all } x,y\in\mathbb{C} \text{ \hfill{(Nagumo equation)}}
\end{equation} 
\begin{equation}\label{H}
\sum_{k=0}^{N-1}f(x+\eta^{2k}y)=\sum_{k=0}^{N-1}f(x+\eta^{2k+1}y) \text{ for all } x,y\in\mathbb{C} \text{ (Haruki equation)}
\end{equation} 
Concretely, he proved that all these equations, when considered over the space of entire functions (i.e., when imposing $f\in H(\mathbb{C})$), are equivalent in the sense that they share the same space of solutions: the algebraic polynomials $p(z)$ of degree $\leq N-1$. To prove this result he used the open mapping theorem for analytic functions. In this section we study the continuous solutions of the equations $\eqref{KNW}$ and $\eqref{H}$. The main tool we need for the proofs is the theorem of Sternfeld and Weit (Theorem \ref{SW} above). 

\begin{theorem} \label{TH} Let $S_{H}$ denote the space of functions $f\in C(\mathbb{C},\mathbb{C})$ which solve $\eqref{H}$. Then 
\[
S_{H}=\{f(z)=\sum_{i=0}^{N-1}\sum_{j=0}^{N-1}a_{i,j}z^i\overline{z}^j :a_{i,j}\in\mathbb{C} \text{ fo all } 0\leq i,j\leq N-1\}.
\]
\end{theorem} 
\begin{proof} We assume, without loss of generality, that $\eta^2=\theta$. Then equation $\eqref{H}$ can be written as 
\begin{equation} \label{oper}
H_y(f)(x)=H_{\eta y}(f)(x) \text{ for all } x,y\in\mathbb{C},
\end{equation}
where $H_y(f)(x)= \frac{1}{N}\sum_{k=0}^{N-1}f(x+\theta^ky)$. We want to identify the set  $S_H=\{f\in C(\mathbb{C},\mathbb{C}): H_y(f)(x)=H_{\eta y}(f)(x) \text{ for all } x,y\in\mathbb{C}\}$. The important observation is that $S_H$ is a closed subspace of $C(\mathbb{C},\mathbb{C})$ which is invariant under the operators $T_{a,b}$, so that $S_H$ must be one of the spaces specified by Theorem \ref{SW}.
Let us, for the sake of completeness, check that $S_H$ is a closed  $T_{a,b}$-invariant subspace of $ C(\mathbb{C},\mathbb{C})$. With this objective in mind, we define the operators $L_y=H_y-H_{\eta y}$ ($y\in \mathbb{C}$). Then, by definition, $S_H=\bigcap_{y\in\mathbb{C}}\ker(L_y)$, so that $S_H$ is a closed linear subspace of $C(\mathbb{C},\mathbb{C})$, since $L_y$ is a continuous linear operator for every $y$. On the other hand, a simple computation shows that
\begin{eqnarray*}
H_y(T_{a,b}(f))(x) &=&   \frac{1}{N}\sum_{k=0}^{N-1}T_{a,b}(f)(x+\theta^ky)= \frac{1}{N}\sum_{k=0}^{N-1}f(a(x+\theta^ky)+b)\\
&=&   \frac{1}{N}\sum_{k=0}^{N-1}f(ax+b+\theta^k(ay)) \\
&=& H_{ay}(f)(ax+b) = T_{a,b}(H_{ay}(f))(x) \text{ for all } x,y,a,b\in\mathbb{C}.
\end{eqnarray*} 
Hence, if $f\in S_H$, then
\[
H_y(T_{a,b}(f))= T_{a,b}(H_{ay}(f))= T_{a,b}(H_{\eta (ay)}(f))= T_{a,b}(H_{a(\eta y)}(f)) = H_{\eta y}(T_{a,b}(f)), 
\]
which means that $T_{a,b}(f)\in S_H$.  It follows that $S_H=\overline{\mathbf{span}_{\mathbb{C}}\left(\bigcup_{k=1}^rA_{[(n_k,m_k)]}\right)}$ for a certain finite set of points 
$\{(n_k,m_k)\}_{k=1}^r\subseteq(\widehat{\mathbb{N}})^2$. In particular, we know that if $z^n\overline{z}^m\in S_H$ then $z^{\alpha}\overline{z}^{\beta}\in S_H$ for all $(\alpha,\beta)$ such that $0\leq \alpha\leq n$ and $0\leq \beta\leq m$. Let us check what functions of the form $z^{\alpha}\overline{z}^{\beta}$ belong to $S_H$.

\noindent \textbf{Claim 1: $z^N,\overline{z}^N \not\in S_H$. } Let us apply the operator $H_y$ to the function $f_N(z)=z^N$:
\begin{eqnarray*}
NH_y(f_N)(x) &= &\sum_{k=0}^{N-1}(x+\theta^ky)^N\\
 &= &\sum_{k=0}^{N-1}\left(\sum_{t=0}^N\binom{N}{t}x^{N-t}\theta^{kt}y^t\right)\\
  &= &\sum_{t=0}^{N}\left(\sum_{k=0}^{N-1}\binom{N}{t}x^{N-t}\theta^{kt}y^t\right)\\
    &= &\sum_{t=0}^{N}\binom{N}{t}x^{N-t}y^t\left(\sum_{k=0}^{N-1}\theta^{kt}\right)\\
    &=& x^NN+y^NN =N(x^N+y^N),
\end{eqnarray*}
since  
\[
\sum_{k=0}^{N-1}\theta^{kt}= \left \{
\begin{array}{cccccc}
\frac{(\theta^t)^N-1}{\theta^t-1}=0 \text{, whenever } t\not\in N\mathbb{Z}\\
 N\text{, whenever } t\in N\mathbb{Z}
 \end{array}\right. .
\]
Thus, if we substitute $y$ by $\eta y$ in the computations above, we get $H_{\eta y}(f_N)(x)=x^N+(\eta y)^N=x^n-y^N$. This obviously implies that $H_y(f_N)\neq H_{\eta y}(f_N)$, so that $z^N\not\in S_H$. Analogous computations will show that $\overline{z}^N\not\in S_H$.

\noindent \textbf{Claim 2: $z^{N-1}\overline{z}^{N-1} \in S_H$. } In this case, we should apply $H_y$ to $g_N(z)= z^{N-1}\overline{z}^{N-1}=(|z|^2)^{N-1}$. As a first step, we observe that 
\begin{eqnarray*}
|x+\theta^ky|^2 &=& (x+\theta^ky)(\overline{x}+\theta^{-k}\overline{y}) \\
&= & x\overline{x}+x\theta^{-k}\overline{y}+\theta^ky\overline{x}+y\overline{y}\\
&= & |x|^2+|y|^2+\theta^ky\overline{x}+\theta^{-k}x\overline{y},
\end{eqnarray*}
so that
\begin{eqnarray*}
&\ & N H_y(g_N)(x) = \sum_{k=0}^{N-1}(|x+\theta^ky|^2)^{N-1}\\
&= & \ \sum_{k=0}^{N-1}(|x|^2+|y|^2+\theta^ky\overline{x}+\theta^{-k}x\overline{y})^{N-1}\\
&= &\ \sum_{k=0}^{N-1}\sum_{\begin{array}{cccccc}
a+b+c+d=N-1\\
a,b,c,d\geq 0
 \end{array}} \binom{a}{a}  \binom{a+b}{b}  \binom{a+b+c}{c}  \binom{N-1}{d} |x|^{2a}|y|^{2b}(\theta^ky\overline{x})^c(\theta^{-k}x\overline{y})^{d}\\
 &= &\ \sum_{\begin{array}{cccccc}
a+b+c+d=N-1\\
a,b,c,d\geq 0
 \end{array}}\left( \binom{N-1}{a,b,c,d}  |x|^{2a}|y|^{2b}(y\overline{x})^c(x\overline{y})^{d} 
\left( \sum_{k=0}^{N-1}(\theta^{c-d})^k\right)
 \right) 
\end{eqnarray*}
Now, $0\leq c,d\leq N-1$ implies that $c-d\in N\mathbb{Z}$ if and only if $c=d$. Hence
\[
 \sum_{k=0}^{N-1}(\theta^{c-d})^k=  \left \{
\begin{array}{cccccc}
0 \text{, if } c \neq d \\
 N\text{, if } c=d 
 \end{array}\right. .
\]
Furthermore, $(y\overline{x})^c(x\overline{y})^{c} = |x|^{2c}|y|^{2c}$, so that 
\begin{eqnarray*}
&\ & N H_y(g_N)(x) \\
&=&   \ 2N\sum_{\begin{array}{cccccc}
a+b+2c=N-1\\
a,b,c\geq 0
 \end{array}} \binom{N-1}{a,b,c,c}  |x|^{2a}|y|^{2b} |x|^{2c}|y|^{2c} 
\end{eqnarray*}
If we substitute $y$ by $\eta y$ in the computations above, the result will not be modified, since $|\eta y|^{2c}=|y|^{2c}$. Hence, $H_{y}(g_N)=H_{\eta y}(g_N)$ and $g_N(z)=z^{N-1}\overline{z}^{N-1}$ belongs to $S_H$.

Taking into account the structure of $S_H$, it follows that 
\begin{itemize}
\item[$(a)$] All functions of the form $f(z)=\sum_{i=0}^{N-1}\sum_{j=0}^{N-1}a_{i,j}z^i\overline{z}^j $ belong to $S_H$, since $z^{N-1}\overline{z}^{N-1}\in S_H$ implies that $A_{[(N-1,N-1)]}\subseteq S_H$.
\item[$(b)$] These are the only elements of $S_H$ since, if $z^n\overline{z}^m\in S_H$ and $\max\{n,m\}>N-1$, then $(0,N)\leq (n,m)$ or $(N,0)\leq (n,m)$, so that $z^N\in S_H$ or $\overline{z}^N\in S_H$, and both things are false, as we have already proved with Claim 1.
\end{itemize}
This ends the proof.
\end{proof}

\begin{corollary} The complex polynomials $p(z)=a_0+a_1z+\cdots+a_{N-1}z^{N-1}$ of degree $\leq N-1$ are the only entire functions which are solutions of the Haruki functional equation $\eqref{H}$.
\end{corollary}

\begin{corollary} 
Let $f\in C(\mathbb{C},\mathbb{R})$. Then $f$ is a solution of $\eqref{H}$ if and only if
\[
f(x+iy)=p(x,y)= \sum_{n=0}^{N-1}\sum_{m=0}^{N-1}a_{n,m}x^ny^m\in\mathbb{R}[x,y]
\]
is an ordinary real polynomial of two real variables with  $\deg_{\max}p(x,y)\leq N-1$.  (Here $x,y$ denote real variables).
\end{corollary}

Let us now consider the continuous solutions of $\eqref{KNW}$. 

\begin{theorem}
Let $S_{KNW}$ denote the space of functions $f\in C(\mathbb{C},\mathbb{C})$ which solve $\eqref{KNW}$. Then 
\[
S_{KNW}=\{a_0+a_1z+\cdots +a_{N-1}z^{N-1}+b_1\overline{z}+\cdots+b_{N-1}\overline{z}^{N-1}:(a_0,\cdots,a_{N-1},b_1,\cdots,b_{N-1})\in\mathbb{C}^{2N-1}\}.
\]
\end{theorem}
\begin{proof}
Obviously, $S_{KNW}$ is closed, since $S_{KNW}=\bigcap_{y\in \mathbb{C}}\ker(H_y-I)$, and it is also affine invariant, since, if $f\in S_{KNW}$, then
\[
H_y(T_{a,b}(f))=T_{a,b}(H_{ay}(f))=T_{a,b}(f),
\] 
so that $T_{a,b}(S_{KNW})\subseteq S_{KNW}$. It follows that $S_{KNW}$ admits a representation of the form $S_H=\overline{\mathbf{span}_{\mathbb{C}}\left(\bigcup_{k=1}^rA_{[(n_k,m_k)]}\right)}$ for a certain finite set of points  $\{(n_k,m_k)\}_{k=1}^r\subseteq(\widehat{\mathbb{N}})^2$. Let us now check, by direct computation, what functions of the form $z^\alpha\overline{z}^\beta$ belong to $S_{KNW}$. 

\noindent \textbf{Claim 1. $z\overline{z}\not\in S_{KNW}$. } Take $g(z)=z\overline{z}$. A direct computation (which we omit here, since it is quite similar to the computations we have already made for the proof of Theorem \ref{TH}), shows that $H_y(g)(x)=N(|x|^2+|y|^2)\neq g(x)$, so that $g\not\in S_{KNW}$. 

\noindent \textbf{Claim 2. $z^N,\overline{z}^N\not\in S_{KNW}$. } We have already shown that, if $f_N(z)=z^N$, then  $H_y(f_N)(x)=x^N+y^N\neq x^N$, so that $z^N\not\in S_{KNW}$. The same computation, taking conjugates, shows that, if $g_N(z)=\overline{z}^N$, then   $H_y(g_N)(x)=\overline{x}^N+\overline{y}^N\neq \overline{x}^N$, so that $\overline{z}^N\not\in S_{KNW}$. 

\noindent \textbf{Claim 3. $z^{N-1}, \overline{z}^{N-1}\in S_{KNW}$. }   Let us apply the operator $H_y$ to the function $f_{N-1}(z)=z^{N-1}$:
\begin{eqnarray*}
NH_y(f_{N-1})(x) &= &\sum_{k=0}^{N-1}(x+\theta^ky)^{N-1}\\
 &= &\sum_{k=0}^{N-1}\left(\sum_{t=0}^{N-1}\binom{N-1}{t}x^{N-1-t}\theta^{kt}y^t\right)\\
  &= &\sum_{t=0}^{N-1}\left(\sum_{k=0}^{N-1}\binom{N-1}{t}x^{N-1-t}\theta^{kt}y^t\right)\\
    &= &\sum_{t=0}^{N-1}\binom{N-1}{t}x^{N-1-t}y^t\left(\sum_{k=0}^{N-1}\theta^{kt}\right)\\
    &=& Nx^{N-1},
\end{eqnarray*}
so that $H_y(f_{N-1})(x)=x^{N-1}$ and $z^{N-1}\in S_{KNW}$. An analogous computation shows that $\overline{z}^{N-1}\in S_{KNW}$.

It follows from Claim 1 that, if $n,m\geq 1$, then $z^n\overline{z}^m\not\in S_{KNW}$ and, from Claims 2 and 3, that $S_{KNW}=A_{[(N-1,0)]}+A_{[(0,N-1)]}$, which is what we wanted to prove. \end{proof}

\begin{corollary} The complex polynomials $p(z)=a_0+a_1z+\cdots+a_{N-1}z^{N-1}$ of degree $\leq N-1$ are the only entire functions which are solutions of the Kakutani-Nagumo-Walsh functional equation $\eqref{KNW}$.
\end{corollary}

\begin{corollary}
Let $f\in C(\mathbb{C},\mathbb{R})$. Then $f$ is a solution of $\eqref{KNW}$ if and only if $f(x+iy)=p(x,y)$ is an harmonic polynomial of degree $\leq N-1$.  (Here $x,y$ denote real variables).
\end{corollary}

%What can be said for the Nagumo equation $\eqref{N}$?

\section{Continuous solutions of Fr\'{e}chet functional equation}

In this section we use  Theorem \ref{P} to compute the solutions $f$ of  the classical Fr\'{e}chet's functional equation,
\begin{equation} \label{FE}
\Delta^{N}_hf(x)=\sum_{k=0}^N\binom{N}{k}(-1)^{N-k}f(x+kh)=0 (x,h \in \mathbb{R}^d), \ \ \text{(Fr\'{e}chet equation with fixed step),}
\end{equation}
under the additional hypothesis that  $f\in C(\mathbb{R}^d,\mathbb{R})$. This equation was introduced in the literature by M. Fr\'{e}chet in 1909, for continuous functions $f:\mathbb{R}\to\mathbb{R}$,  as a particular case of the functional equation  
\begin{equation}\label{fre}
\Delta_{h_1h_2\cdots h_{N}}f(x)=0 \ \ (x,h_1,h_2,\dots,h_{N}\in \mathbb{R}), 
\end{equation}
where $\Delta_{h_1h_2\cdots h_s}f(x)=\Delta_{h_1}\left(\Delta_{h_2\cdots h_s}f\right)(x)$, $s=2,3,\cdots$. In particular, after Fr\'{e}chet's 
seminal paper \cite{frechet}, the solutions of \eqref{fre} are named ``polynomials'', since it is known that, under very mild regularity conditions on $f$, if $f:\mathbb{R}\to\mathbb{R}$ satisfies \eqref{fre}, then $f(x)=a_0+a_1x+\cdots a_{N-1}x^{N-1}$ for all $x\in\mathbb{R}$ and certain constants $a_i\in\mathbb{R}$. For example, in order to have this property, it is enough for $f$ being locally bounded \cite{frechet}, \cite{almira_antonio}, but there are stronger results \cite{ger1}, \cite{kuczma1}, \cite{mckiernan}, \cite{popa_rasa}. The equation \eqref{fre} can be studied for functions $f:X\to Y$  whenever $X, Y$ are two  $\mathbb{Q}$-vector spaces and the variables $x,h_1,\cdots,h_{N}$ are assumed to be elements of $X$:
\begin{equation}\label{fregeneral}
\Delta_{h_1h_2\cdots h_{N}}f(x)=0 \ \ (x,h_1,h_2,\dots,h_{N}\in X) \ \ \text{(Fr\'{e}chet equation with variable step).}
\end{equation}
In this context, the general solutions of \eqref{fregeneral} are characterized as functions of the form $f(x)=A_0+A_1(x)+\cdots+A_{N-1}(x)$, where $A_0$ is a constant and $A_k(x)=A^k(x,x,\cdots,x)$ for a certain $k$-additive symmetric function $A^k:X^k\to Y$ (we say that $A_k$ is the diagonalization of $A^k$). In particular, if $x\in X$ and $r\in\mathbb{Q}$, then $f(rx)=A_0+rA_1(x)+\cdots+r^{N-1}A_{N-1}(x)$. Furthermore, it is known that $f:X\to Y$ satisfies \eqref{fregeneral} if and only if it satisfies 
\begin{equation}\label{frepasofijo}
\Delta_{h}^{N}f(x):=\sum_{k=0}^{N}\binom{N}{k}(-1)^{N-k}f(x+kh)=0 \ \ (x,h\in X).
\end{equation}
A proof of this fact follows directly from Djokovi\'{c}'s Theorem \cite{Dj} (see also \cite[Theorem 7.5, page 160]{HIR}, \cite[Theorem 15.1.2., page 418]{kuczma}), which states that  the operators $\Delta_{h_1 h_2\cdots h_s}$ satisfy the equation
\begin{equation}\label{igualdad}
\Delta_{h_1\cdots h_s}f(x)=
\sum_{\epsilon_1,\dots,\epsilon_s=0}^1(-1)^{\epsilon_1+\cdots+\epsilon_s}
\Delta_{\alpha_{(\epsilon_1,\dots,\epsilon_s)}(h_1,\cdots,h_s)}^sf(x+\beta_{(\epsilon_1,\dots,\epsilon_s)}(h_1,\cdots,h_s)),
\end{equation}
where $$\alpha_{(\epsilon_1,\dots,\epsilon_s)}(h_1,\cdots,h_s)=(-1)\sum_{r=1}^s\frac{\epsilon_rh_r}{r}$$ and $$\beta_{(\epsilon_1,\dots,\epsilon_s)}(h_1,\cdots,h_s)=\sum_{r=1}^s\epsilon_rh_r.$$  

In this section we impose $f\in C(\mathbb{R}^d,\mathbb{R})$ and, with very different arguments to those originally used by Fr\'{e}chet, we give a new proof of 
Fr\'{e}chet's  result for this case. Furthermore, we use Djokovi\'{c}'s Theorem to simplify the computations. 

\begin{theorem} Let $S_F$ denote the set of functions  $f\in C(\mathbb{R}^d,\mathbb{R})$ which solve $\eqref{FE}$ (or, equivalently, $\eqref{fregeneral}$ with $X=\mathbb{R}^d$, $Y=\mathbb{R}$). Then
\[
S_F=\Pi_{N-1}^d.
\]
\end{theorem}
\begin{proof}
We first prove that $S_F$ is a closed subspace of $C(\mathbb{R}^d,\mathbb{R})$ which is invariant by the transformations of the form $S_{a,b}(f)(x)=f(a\cdot x+b)$ for all $a,b\in\mathbb{R}^d$. Obviously, for each $h\in\mathbb{R}^d$, $\Delta_h^N:C(\mathbb{R}^d,\mathbb{R})\to C(\mathbb{R}^d,\mathbb{R})$ is a continuous linear operator, so that $S_F=\bigcap_{h\in \mathbb{R}^d}\ker \Delta_h^N$ is a closed linear subspace of $C(\mathbb{R}^d,\mathbb{R})$. Let us now take $a,b\in\mathbb{R}^d$ and assume that $f\in S_F$. Then, for each $h\in\mathbb{R}^d$ we have that
\begin{eqnarray*}
\Delta_h^N(S_{a,b}(f)) &=& \sum_{k=0}^N\binom{N}{k}(-1)^{N-k}S_{a,b}(f)(x+kh)\\
&=& \sum_{k=0}^N\binom{N}{k}(-1)^{N-k}f(a\cdot x+b+ k(a\cdot h)) \\
&=& \Delta_{a\cdot h}^N(f)(a\cdot x+b)\\
&=&S_{a,b}( \Delta_{a\cdot h}^N(f))(x) =0, 
\end{eqnarray*}
so that $S_{a,b}(f)\in S_F$, and $S_{a,b}(S_F)\subseteq S_F$, which is what we wanted to prove. It follows from Theorem \ref{P} that 
\[
S_F=\overline{\mathbf{span}_{\mathbb{R}}\left(\bigcup_{k=1}^rB_{[n_k]}\right)}
\] 
for a certain finite set of points $\{n_k=(n_{k,1},\cdots,n_{k,d})\}_{k=1}^r\subseteq(\widehat{\mathbb{N}})^d$. 

A direct computation shows that, if $f_i(x)=x_i^N$ and $e_i=(0,0,\cdots, 1^{(i\text{-th position})},0,\cdots,0)$, then $ \Delta_{e_i}^N(f_i)(x)=N!\neq 0$, so that $f_i\not\in S_F$, $i=1,\cdots,d$. Thus, if $\alpha=(\alpha_1,\cdots,\alpha_d)\in\mathbb{N}^d$ satisfies $\max_{1\leq i\leq d}\alpha_i\geq N$, then $x^{\alpha}\not\in S_F$ since otherwise we should have that $(0,0,\cdots, N^{(i\text{-th position})},0,\cdots,0)\leq \alpha$ for some $i$, so that $x_i^N\in S_F$ for some $i$, which is false.This proves that all elements of $S_F$ are ordinary polynomials. Indeed, $S_F\subseteq \mathbf{span}_{\mathbb{R}}\{x^\alpha:\max\alpha_i\leq N-1\}$. We want to show a deeper result, since we want to identify completely $S_F$. Concretely, we prove that $S_F=\Pi_{N-1}^d$, the space of polynomials with total degree $\leq N-1$. We divide the proof in two steps:

\noindent \textbf{Step 1: $\Pi_{N-1}^d\subseteq S_F$.} This is a direct consequence of the fact that the operator $\Delta_h$ reduces the total degree of a polynomial in at least one unity, since
\begin{eqnarray*}
\Delta_hx^{\alpha} &=& (x+h)^{\alpha}-x^{\alpha} =\prod_{i=1}^d(x_i+h_i)^{\alpha_i} -\prod_{i=1}^d x_i^{\alpha_i} \\
&=& \prod_{i=1}^d(x_i^{\alpha_i}+\sum_{t_i=0}^{\alpha_i-1}\binom{\alpha_i}{t_i}h_i^{\alpha_i-t_i}x_i^{t_i}) -\prod_{i=1}^d x_i^{\alpha_i}\in \Pi_{|\alpha|-1}^d.
\end{eqnarray*}

\noindent \textbf{Step 2: $S_F \subseteq \Pi_{N-1}^d$.} To prove this result, we use Djokovi\'{c}'s Theorem or, what is the same, the equivalence of the equations $\eqref{FE}$  and $\eqref{fregeneral}$ (with $X=\mathbb{R}^d$, $Y=\mathbb{R}$). Obviously, due to the structure of $S_F$, it is enough to prove that, if $|\alpha|\geq N$, then $x^\alpha\not\in S_F$. 

Let $i\in \{1,\cdots,d\}$ be fixed and let $e_i=(0,0,\cdots, 1^{(i\text{-th position})},0,\cdots,0)$, $\alpha=(\alpha_1,\cdots,\alpha_d)\in\mathbb{N}^d$, and $f_{\alpha}(x)=x^\alpha$. Then, an easy computation shows that, for each $k\in\mathbb{N}$, 
\[
\Delta_{e_i}^kf_{\alpha}(x)= (\Delta_{1}^kx_i)x_1^{\alpha_1}x_2^{\alpha_2}\cdots x_{i-1}^{\alpha_{i-1}}x_{i+1}^{\alpha_{i+1}}\cdots x_{d}^{\alpha_{d}}.
\]
Furthermore, it is well known that the operators $\Delta_h$, $\Delta_k$ commute under composition, so that, if $|\alpha|\geq N$, then 
\[
\Delta_{e_1}^{\alpha_1}\Delta_{e_2}^{\alpha_2}\cdots \Delta_{e_d}^{\alpha_d}(f_{\alpha})(x) = 
(\Delta_{1}^{\alpha_1}x_1^{\alpha_1}) (\Delta_{1}^{\alpha_2}x_2^{\alpha_2}) \cdots (\Delta_{1}^{\alpha_d}x_d^{\alpha_d})= \alpha_1!\alpha_2!\cdots\alpha_d!\neq 0, 
 \]
which implies that $x^{\alpha}$ does not vanishes identically for the operators of the form $\Delta_{h_1h_2\cdots h_{s}}$ with $s=|\alpha|\geq N$. In particular, $x^{\alpha}\not\in S_F$, which is what we wanted to prove.
\end{proof}

 \bibliographystyle{amsplain}

\begin{thebibliography}{99}
\bibitem{almira_antonio} {\sc J. M. Almira, A. J.  L\'{o}pez-Moreno, } On solutions of the Fr\'{e}chet functional equation, J. Math. Anal. Appl.  \textbf{332} (2007)  1119-1133.

\bibitem{almira_invariantes} {\sc J. M. Almira, } Montel's Theorem and subspaces of distributions which are $\Delta^m$-invariant, to appear in \emph{Numer. Funct. Anal. Optimization}  (2013) 
DOI:10.1080/01630563.2013.813537.

\bibitem{almira_montel_vv}  {\sc J. M. Almira, Kh. F. Abu-Helaiel, } On Montel's Theorem in several variables, Manuscript, 2013. Available at http://arxiv.org/abs/1310.3378.

%\bibitem{anselone} {\sc P. M. Anselone, J. Korevaar, } Translation invariant subspaces of finite dimension, \emph{Proc. Amer. Math. Soc. } \textbf{15} (1964), 747-752.

%\bibitem{ron} {\sc C. de Boor, A. Ron} Polynomial ideals and multivariate splines, 1989.

\bibitem{czerwik} {\sc S. Czerwik, } \textit{Functional equations and inequalities in several variables,} World Scientific, 2002.


\bibitem{Dj} {\sc D. Z. Djokovi\'{c}, } A representation theorem for $(X_1-1)(X_2-1)\cdots(X_n-1)$ and its applications, \emph{Ann. Polon. Math.} \textbf{22} (1969/1970) 189-198.

\bibitem{frechet} {\sc M. Fr\'{e}chet, }   Une definition fonctionelle des polynomes, Nouv. Ann. 9 (1909), 145-162.

\bibitem{ger1}  {\sc R. Ger}, On some properties of polynomial functions, Ann. Pol. Math. \textbf{25 }(1971) 195-203.

\bibitem{ger} {\sc R. Ger, } On extensions of polynomial functions, Results in Mathematics \textbf{26} (1994), 281-289.


%\bibitem{gohberg} {\sc I. Gohberg, P. Lancaster, L. Rodman, } \emph{Invariant subspaces of matrices with applications, } Classics in Applied mathematics \textbf{51}, 
%S.I.A.M.,  2006.

\bibitem{haruki1} {\sc S. Haruki, } On the mean value property of harmonic and complex polynomials, Proc. Japan Acad. Ser. A, \textbf{57}  (1981) 216-218. 

\bibitem{haruki2} {\sc S. Haruki, } On the theorem of S. Kakutani-M. Nagumo and J.L. Walsh for the mean value property of harmonic and complex polynomials, Pacific J. Math. \textbf{94} (1) (1981) 113-123.

\bibitem{haruki3} {\sc S. Haruki, } On two functional equations connected with a mean-value property of polynomials, Aequationes Math. \textbf{6} (1971) 275-277.

\bibitem{HIR} {\sc D. H. Hyers, G. Isac, T. M. Rassias, } \emph{Stability of functional equations in several variables, } Birkh\"{a}user, 1998. 

\bibitem{kakutani_nagumo} {\sc S. Kakutani, M. Nagumo, } About the functional equation $\sum_{v=0}^{n-1}f(z+e^{(\frac{2 v \pi}{2})i})=nf(z)$, Zenkoku Shij\^{o} Danwakai, \textbf{66} (1935) 10-12 (in Japanese).

\bibitem{kuczma1} {\sc M. Kuczma}, On measurable functions with
vanishing differences, Ann. Math. Sil. \textbf{6} (1992) 42-60.

\bibitem{kuczma} {\sc M. Kuczma}, \textit{An introduction to the theory of functional equations and inequalities, Second Edition,} Birkh\"{a}user Verlag, 2009.

%\bibitem{lefranc} {\sc M. Lefranc, } Analyse spectrale sur $\mathbb{Z}^n$,  C. R. Acad. Sci. Par\'{\i}s, \textbf{246}  (1958) 1951-1953.

%\bibitem{leland} {\sc K. O. Leland, } Finite dimensional translation invariant spaces, \emph{Amer. Math. Monthly} \textbf{75} (1968) 757-758.

\bibitem{mckiernan}  {\sc M. A. Mckiernan, } On vanishing n-th ordered differences and Hamel bases, Ann. Pol. Math. \textbf{19} (1967) 331-336.

%\bibitem{montel} {\sc P. Montel, } Sur quelques extensions d'un th\'{e}or\`{e}me de Jacobi, \emph{Prace Matematyczno-Fizyczne} \textbf{44} (1) (1937) 315-329.

\bibitem{pinkus_TDI} {\sc A. Pinkus, } TDI-subspaces of $C(\mathbb{R}^d)$ and some density problems for neural networks, Journal of Approximation Theory
 \textbf{85} (3) (1996) 269-287.
 
 
\bibitem{popa_rasa} {\sc D. Popa, I. Rasa, }  The Fr\'{e}chet functional equation with application to the stability of certain operators, \emph{Journal of Approximation Theory}  \textbf{164}  (1)  (2012) 138-144.

%\bibitem{prager} {\sc W. Prager, J. Schwaiger } Generalized polynomials in one and in several variables, Mathematica Pannonica  \textbf{20} (2) (2009) 189-208.

\bibitem{sternfeld_weit} {\sc Y. Sternsfeld, Y. Weit, } Affine invariant subspaces of $C(\mathbb{C})$, Proc. Amer. Math. Soc. \textbf{107} (1989) 231-236.

%\bibitem{laszlo} {\sc L. Sz\'{e}kelyhidi, } \emph{Discrete Spectral Synthesis and its Applications, } Springer Monographs on Mathematics, Springer, 2006.

\bibitem{walsh} {\sc J. L. Walsh, } A mean value theorem for polynomials and harmonic polynomials, Bull. Amer. Math. Soc. \textbf{42} (1936) 923-930. 

\end{thebibliography}

%%%  ==============================================================

\bigskip

\footnotesize{J. M. Almira and Kh. F. Abu-Helaiel

Departamento de Matem\'{a}ticas. Universidad de Ja\'{e}n.

E.P.S. Linares,  C/Alfonso X el Sabio, 28

23700 Linares (Ja\'{e}n) Spain

Email: jmalmira@ujaen.es; kabu@ujaen.es }

%Phone: (34)+ 953648503

%Fax: (34)+ 953648575}

%Phone: (34)+ 953648503

%Fax: (34)+ 953648575}

\end{document}